\documentclass[12pt,leqno]{amsart}
\usepackage{amsfonts,amsthm,amsmath,mathtools,xcolor,comment,amssymb}

\theoremstyle{plain}
\newtheorem{thm}{Theorem}[section]

\newtheorem{lem}{Lemma}[section]

\theoremstyle{definition}
\newtheorem{df}{Definition}[section]
\newtheorem{rem}{Remark}[section]

\newcommand{\FF}{\mathbb{F}}

\newcommand{\CC}{\mathbb{C}}

\newcommand{\I}{\mathrm{I}}
\newcommand{\II}{\mathrm{II}}

\DeclareMathOperator{\supp}{supp}

\DeclareMathOperator{\wt}{wt}

\DeclareMathOperator{\diag}{diag}

\def\bm#1{\mathbf{#1}}

\begin{document}

\title[Jacobi polynomials, invariant rings, and designs]{Jacobi polynomials, invariant rings, and generalized $t$-designs}

\author[Chakraborty]{Himadri Shekhar Chakraborty}
\address{Department of Mathematics, Shahjalal University of Science and Technology, Sylhet 3114, Bangladesh}
\email{himadri-mat@sust.edu}

\author[Hamid]{Nur Hamid*}
\address{Department of Mathematics Education, Nurul Jadid University, Paiton, Probolinggo, Indonesia}
\email{nurhamid@unuja.ac.id}

\author[Miezaki]{Tsuyoshi Miezaki}
\address{Faculty of Science and Engineering, Waseda University, Tokyo 169-8555, Japan}
\email{miezaki@waseda.jp}

\author[Oura]{Manabu Oura}
\address
{Faculty of Mathematics and Physics, Kanazawa University, Ishikawa 920-1192, Japan}
\email{oura@se.kanazawa-u.ac.jp} 

\thanks {*Corresponding author}

\date{}
\maketitle

\begin{abstract}
	In the present paper, 
	we provide results that 
	relate the Jacobi polynomials in genus $g$. 
	We show that 
	if a code is $t$-homogeneous 
	that is, 
	the codewords of the code for every given weight hold a $t$-design,
	then 
	its Jacobi polynomial in genus $g$ with 
	composition $T$ with $|T|\leq t$ 
	can be obtained from its weight enumerator in genus~$g$
	using the polarization operator. 
	Using this fact, 
	we investigate the invariant ring, which relates 
	the homogeneous Jacobi polynomials of the binary codes in genus $g$. 
	Specifically, 
	the generators of the invariant ring appearing for $g=1$ are obtained.
	Moreover, we define the split Jacobi polynomials in genus~$g$
	and obtain the MacWilliams type identity for it. 
	A split generalization for higher genus cases 
	of the relation between the
	Jacobi polynomials and weight enumerator
	of a $t$-homogeneous code also given.
\end{abstract}

{\small
\vspace{-0.15in}

\quad
}


\section{Introduction}

The coding theoretical analogue of Jacobi forms (cf.~\cite{EZ}) are 
known as Jacobi polynomials, representing a significant
generalization of the weight enumerator of a code.
Note that one of most remarkable generalization of 
Siegel modular forms (cf.~\cite{K}) in lattice theory (cf.~\cite{CS}) 
are Jacobi forms. 
The notion of Jacobi polynomials for codes 
was first introduced by Ozeki \cite{Ozeki}
to describe the transformation formula 
for the Jacobi polynomials of a code. 
This formula includes the MacWilliams identity (cf.~\cite{MacWilliams}) 
as a special case for the weight enumerator of the code.
Later, Bonnecaze et al.~\cite{BMS} gave the notion of 
Jacobi polynomial in the sense of coordinates.
They pointed out that, in some cases, 
the Jacobi polynomials can be determined uniquely
from the weight enumerators with the help of polarization operator.
Note that Ozeki's definition depends on a reference vector,
while Bonnecaze el al. depends on a partition of the support.

Many authors studied 
several generalizations of the Jacobi polynomials 
in coding theory.
Among them the higher genus generalizations were
studied in~\cite{CM2021, CMO2022, HO}
and split type generalizations were 
studied in~\cite{CIT2024}. 
Moreover, Cameron~\cite{Cameron2009} gave a 
celebrated generalization 
of the classical $t$-designs
that we prefer to call as the generalized $t$-designs. 
In a recent work~\cite{CMOT2023}, 
the concept of Jacobi polynomials 
attached to multiple reference vectors 
was introduced. 
This approach provided a design theoretical application of 
these Jacobi polynomials through the use of generalized $t$-designs.

Invariant theory (cf.~\cite{NRS}) plays an important role in the 
study of some special codes,
such as binary self-dual doubly-even codes,
which are known as Type~$\II$ codes.
Gleason~\cite{Gleason} proved that
the invariant ring over~$\CC$
under the group of order~$192$ generated by
the matrices:
\[
	\frac{1}{\sqrt{2}}
	\begin{pmatrix*}[r]
		1 & 1\\
		1 & -1
	\end{pmatrix*}
	\mbox{ and }
	\begin{pmatrix}
		1 & 0\\
		0 & i
	\end{pmatrix},
\]
can be represented by the weight enumerators of
Type~$\II$ codes.
Numerous articles in algebraic combinatorics 
have explored various generalizations of this
result. 
For instance, 
Duke~\cite{Duke} and Runge's~\cite{Runge}
study on Siegel modular forms 
investigates the weight enumerators of Type~$\II$ codes
in genus~$g$. 
Eventually, Bannai and Ozeki's~\cite{BannaiOzeki}
work on Jacobi forms determined the Molien series (cf.~\cite{Molien}) 
of the invariant ring of Jacobi polynomials;
see also~\cite{BannaiOzekiTeranishi}.
By defining a new map from the space of Jacobi polynomials into the space of Jacobi forms,
they also extended Brou\'e-Enguehard (cf.~\cite{BE1972}) correspondence.

In this paper, 
we studied Jacobi polynomials of codes in genus $g$
in the sense of coordinate positions. 
Throughout this study, 
by codes we mean binary linear codes.
We show how the polarization operator 
may be used to obtain the Jacobi polynomials  
of a $t$-homogenous code in genus~$g$
with composition~$T$ such that $|T|\leq t$. 
Using this fact, 
we investigate the invariant ring 
of Jacobi polynomials of 
Type~$\II$ codes in genus $g$. 
For the computation, 
we give the generators for 
the invariant ring of Jacobi polynomials 
of Type~$\II$ codes in genus 1.
We also show that these generators are enough 
to generate the invariant ring of Jacobi polynomials
of Type~$\II$ codes. 
In genus 2, 
we obtain the generators for the invariant ring of 
Jacobi polynomials for Type~$\II$ codes of lengths up to~$24$.  
Additionally, we introduce the split Jacobi polynomials in genus~$g$ 
and derive the MacWilliams-type identity for them. 
As an application of generalized $t$-designs,
we also provide a split generalization of the relationship 
between Jacobi polynomials and the weight enumerator 
of a $t$-homogeneous code for higher genus cases.

This paper is organized as follows. 
In Section~\ref{Sec:Preli}, 
we discuss definitions and the basic properties of codes 
and generalized $t$-designs
that are needed to understand this paper. 
In Section~\ref{Sec:JacobiPoly}, 
we give the MacWilliams type identity (Theorem~\ref{Thm:JacMacGen_g}) 
for Jacobi polynomials in genus~$g$. 
We also show how a polarization operator acts 
to obtain the Jacobi polynomials in genus~$g$ 
(Theorems~\ref{Thm:gthJacPolythomo}, \ref{Thm:1-homo_g}, \ref{Thm:t-homo_g}).
In Section~\ref{Sec:SplitJacobiPoly},
we give the MacWilliams type identity (Theorem~\ref{Thm:gthSplitJacobiMacWilliams})
for the split Jacobi polynomials in genus~$g$.
We also observe (Theorems~\ref{Thm:JacToDesign}, \ref{Thm:Main1}, \ref{Thm:SplitMain2}) how polarization operator acts to obtain the split Jacobi polynomials 
in genus~$g$ attached to multiple sets of coordinate places of a code.
In Section~\ref{Sec:InvRing},
we construct the group under which
the Jacobi polynomials in genus~$g$ of Type~$\II$
codes are invariant (Theorem~\ref{Thm:JacInvUnderGg}).
As a generalization of~\cite[Theorem 1.1]{BannaiOzeki},
we also determine the dimension formulae 
(Theorems~\ref{Thm:MolienJacG11}, \ref{Thm:MolienJac21}) 
of the invariant ring 
of Jacobi polynomials of Type~$\II$ codes in genus~$1$ and
genus~$2$.
In Section~\ref{Sec:GenInvRing},
we evaluate the generators of the invariant ring
of Jacobi polynomials of Type~$\II$ codes in genus~$1$ (Theorem~\ref{Thm:JacGenJ8J24}) and
genus~$2$.

All computer calculations in this paper were done with the help of
SageMath~\cite{SageMath}.

\section{Preliminaries}\label{Sec:Preli}

\subsection{Binary linear codes}
  
Let $\FF_{2}^{n}$ be the vector space of dimension $n$ over 
the binary field $\FF_{2}$ with two elements $0$ and $1$.
The elements of $\FF_{2}^{n}$ are known as \emph{vectors}.
The \emph{weight} of a vector
${u} = (u_1,\dots, u_n)\in \FF_{2}^{n}$ 
is denoted by~$\wt({u})$ and
defined to be the number of $i$'s such that $u_{i} \neq 0$.
Let ${u} = (u_1,\dots, u_n)$ 
and ${v} = (v_1,\dots,v_n)$ 
be the vectors in $\FF_{2}^{n}$. 
Then the \emph{inner product} of two vectors 
${u},{v} \in \FF_{2}^{n}$ 
is given by
\[
	{u} \cdot {v} 
	:= 
	u_{1} v_{1} + \dots + u_{n}v_{n}.
\]

An $\FF_{2}$-linear code or a binary code or simply a code of length~$n$ 
is a vector subspace of $\FF_{2}^{n}$. 
The elements of a code are called \emph{codewords}. 
A binary code for which the weights of all the codewords
are even (resp. divisible by~$4$) is called \emph{even}
(resp. \emph{doubly-even}) codes.
The \emph{dual code} of a code~$C$ 
of length~$n$ is defined by
\[
	C^\perp 
	:= 
	\{
	{v}\in \FF_{2}^{n} 
	\mid 
	{u} \cdot {v} = 0 
	\text{ for all } 
	{u}\in C
	\}. 
\]
An code $C$ is called \emph{self-dual} if $C = C^\perp$. 
It is well known that 
the length~$n$ of a self-dual code 
is even and the dimension is $n/2$.
To study self-dual codes in detail, we refer the readers 
to~\cite{BMS1972, Gleason, NRS}. 
A self-dual code~$C$ is called \emph{Type}~$\I$ or \emph{Type}~$\II$ 
if~$C$ is even or doubly-even, respectively.

\begin{df}
	Let $C$ be a code of length~$n$.
	Then the \emph{genus~$g$ weight enumerator} of $C$ 
	is defined as
	\[
		W_{C}^{(g)}(\{x_{a}\}_{a \in \FF_{2}^{g}})
		:=
		\sum_{u_{1},\ldots,u_{g} \in C}
		\prod_{a \in \FF_{2}^{g}}
		x_{a}^{n_{a}(u_{1},\ldots,u_{g})},
	\]
	where,
	$n_{a}(u_{1},\ldots,u_{g})$ is the number of coordinate places $i$ 
	such that $a = (u_{1i},\ldots,u_{gi})$. 
\end{df}

Let $[n] := \{1,2,\ldots, n\}$ represent the coordinate positions of
a vector in~$\FF_{2}^{n}$. 
The Jacobi polynomial of a code 
attached to a set $T \subseteq [n]$
were defined in~\cite{BMS} as follows:

\begin{df}\label{Def:JacOne}
	Let $C$ be a code of length~$n$. Then 
	the \emph{Jacobi polynomial} attached to a set $T \subseteq [n]$ of 
	coordinate places of the code $C$ is defined as:
	\[
		J_{C,T}(w,z,x,y) 
		:=
		\sum_{u\in C}
		w^{m_0({u})}z^{m_1({u})}x^{n_0({u})}y^{n_1({u})},
	\]
	where $m_i({u})$ is the Hamming composition of 
	$u$ on $T$ and $n_i({u})$ is the Hamming composition of 
	${u}$ on $[n]\backslash T$.
\end{df}

It well-known that the Jacobi polynomial~$J_{C,T}(x,y)$ of a code~$C$ 
satisfies the following MacWilliams identity (see~\cite{BMS}):
\[
	J_{C^\perp,T}(w,z,x,y)
	=
	\frac{1}{|C|}
	J_{C,T}(w+z,w-z,x+y,x-y).
\]

\subsection{Generalized $t$-designs}

Let $t$, $k$, $\lambda$ be the integers such that 
$\lambda > 0$ and $k > t > 0$. 
Again let 
$\bm{k} := (k_{1},\ldots,k_{\ell})$ such that 
$k = \sum_{i = 1}^{\ell} k_{i}$,
$\bm{v} := (v_{1},\ldots,v_{\ell})$ 
such that $v_{i} \geq k_{i}$ for all $i$.
Let $\bm{X} := (X_{1},\ldots,X_{\ell})$, 
where $X_{i}$'s are pairwise disjoint sets with $|X_{i}| = v_{i}$ for all $i$ and  
$$\mathcal{B} \subseteq \binom{X_{1}}{k_{1}} \times \cdots \times \binom{X_{\ell}}{k_{\ell}}.$$

\begin{df}	
	A $t$-$(\bm{v},\bm{k},\lambda)$ \emph{design} or a \emph{generalized} $t$-\emph{design} (in short) is a pair $\mathcal{D} := (\bm{X},\mathcal{B})$ 
	with the following property:
	if $\bm{t} := (t_{1},\ldots,t_{\ell})$ such that $t = \sum_{i = 1}^{n} t_{i}$ satisfying $0 \leq t_{i} \leq k_{i}$ for all $i$, 
	then for any choice of~$\bm{T} := (T_{1},\ldots,T_{\ell})$ with $T_{i} \in \binom{X_{i}}{t_{i}}$ for all $i$, 
	there are precisely~$\lambda$ members 
	$\bm{K} := (K_{1},\ldots,K_{\ell}) \in \mathcal{B}$ for which $T_{i} \subseteq K_{i}$ for all $i$.
\end{df}

Note that in the case when $\bm{k} = (k)$ and $\bm{v} = (v)$, 
this is precisely the definition of a combinatorial 
$t$-$(v,k,\lambda)$ design or a $t$-\emph{design} (in short). 
Moreover, if $\bm{k} = (0,\ldots,0)$ or $\lambda = 0$,
we call the designs as \emph{trivial generalized $t$-designs}. 
Let us recall the concept of $t$-homogeneous code. 
We say a code $C$ is \textit{$t$-homogeneous} 
if the codewords of every given weight hold a $t$-design.
In case $C$ is $t$-homogeneous, 
the Jacobi polynomial $J_{C,T}$ does not depend on $T$ for $|T|=t$
(see~\cite{BMS}).
For this case, we write $J_{C,t}$ instead of $J_{C,T}$.

Now we form a generalized $t$-design from a code as follows.
Let $\bm{v} = (v_{1},\ldots,v_{\ell})$ such that 
$\sum_{i = 1}^{\ell} v_{i} = n$ and
$\bm{X} = (X_{1},\ldots,X_{\ell})$ of pairwise disjoint sets $X_{i} \subseteq [n]$ with $|X_{i}| = v_{i}$.
Again let ${u} = (u_{1},\ldots,u_{n})\in \FF_{2}^{n}$. 
Then for $X \subseteq [n]$, we define
\begin{align*}
	\supp_{X}({u}) &:= \{i \in X \mid u_{i} \neq 0\},\\
	\bm{K}({u}) &:= (\supp_{X_{1}}({u}),\ldots,\supp_{X_{\ell}}({u})),\\
	\wt_{X}({u}) &:= |\supp_{X}({u})|.
\end{align*}

Again for any positive integer $k$, let 
$\bm{k} = (k_{1},\ldots,k_{\ell})$ such that 
$\sum_{i = 1}^{\ell} k_{i} = k$.
Let $C$ be a code of length~$n$. Then
\begin{align*}
	C_{\bm{k}} 
	& := 
	\{{u} \in C \mid \wt_{X_{i}}({u}) = k_{i} \mbox{ for all } i\},\\  \mathcal{B}(C_{\bm{k}}) 
	& := 
	\{\bm{K}({u}) \mid {u} \in C_{\bm{k}}\}.
\end{align*}
Since $C$ is a binary code of length~$n$, 
$\mathcal{B}(C_{\bm{k}})$ is not a multi-set. We say $C_{\bm{k}}$ is a $t$-$(\bm{v},\bm{k},\lambda)$ design if $(\bm{X},\mathcal{B}(C_{\bm{k}}))$ is a $t$-$(\bm{v},\bm{k},\lambda)$ design. A code is called an $\ell$-\emph{th} $t$-\emph{homogeneous} if the codewords of every given weight~$k$ hold a 
$t$-$(\bm{v},\bm{k},\lambda)$ design.

\section{Jacobi Polynomials}\label{Sec:JacobiPoly}


The purpose of this section is as follows: 
If a code $C$ is $t$-homogeneous then 
its Jacobi polynomial in genus~$g$ with $T$ with $|T|\leq t$ 
can be obtained from the genus~$g$ weight enumerator of~$C$
using the polarization operator. 
This result is a generalization of~\cite[Theorems $3$ and $4$]{BMS}.

The higher genus generalization of Jacobi polynomials of a code 
was given in~\cite{HO}.
Here, we give the generalization in the sense of coordinate places. 

\begin{df}  \label{gen_jacobi}
Let $C$ be a code of length~$n$.
Then the $g$-th Jacobi polynomial of~$C$ attached to a set $T\subseteq [n]$ 
is defined by
\[
	J_{C,T}^{(g)} (\lbrace y_a, x_a \rbrace_{a \in \mathbb{F}_2^g})
	:= 
	\sum_{u_1,...,u_g \in C } 
	\prod_{a \in \mathbb{F}_2^g} 
	y_a^{m_a (u_{1},\ldots,u_{g})}\, x_a^{n_a (u_{1},\ldots,u_{g})}
\]
where $m_a(u_{1},\ldots,u_{g})$ (resp. $n_a (u_{1},\ldots,u_{g})$) 
is the number of~$i$ such that $a = (u_{1i},\ldots,u_{gi}) \in \FF_{2}^{g}$ 
on $T$ (resp. $[n] \backslash T$).
\end{df}

\begin{rem} \label{rem:relation_on_Jacobi}
	The following relations are immediate from Definition~\ref{gen_jacobi}:
	\begin{itemize}
		\item [(i)]
		if $T$ is an empty set, then
		$J_{C,T}^{(g)}= W_C^{(g)}.$
		
		\item [(ii)]
		$J_{C,T}(\lbrace y_a, x_a \rbrace_{a \in \mathbb{F}_2^g})= J_{C,[n]\backslash T}(\lbrace x_a, y_a \rbrace_{a \in \mathbb{F}_2^g})$.
	\end{itemize}
\end{rem}

The following MacWilliams type identity is taken from \cite{HO}.
To keep the variables notation, 
we change the variables in \cite{HO} from $y_{(b_1 \ \cdots \ b_g \ 0)}$ to $x_{(b_1 \ \cdots \ b_g)}$ and 
$y_{(b_1 \ \cdots \ b_g \ 1)}$ to $y_{(b_1 \ \cdots \ b_g)}$.

\begin{thm}[MacWilliams type identity]\label{Thm:JacMacGen_g}
	Let~$C$ be a code of length of length~$n$.
	Let $J_{C,T}^{(g)}(\lbrace y_a, x_a \rbrace_{a \in \mathbb{F}_2^g})$
	be the $g$-th Jacobi polynomial of~$C$ attached to a set~$T \subseteq [n]$.
	Then
	\[
		\begin{aligned} 
			J_{C^{\perp},T}^{(g)}
			& (\lbrace y_a, x_a \rbrace_{a \in \mathbb{F}_2^g})\\ 
			& =  
			\frac{1}{|C|^g}\ 
			J_{C,T}^{(g)} 
			& \left( \left\lbrace  \sum_{b \in \FF_{2}^{g}} (-1)^{a \cdot b} y_b ,
			\sum_{b \in \FF_{2}^{g}} (-1)^{a \cdot b} x_b  \right\rbrace_{a \in \mathbb{F}_2^g} \right).
		\end{aligned}
	\]
\end{thm}



Let $C$ be a code of length~$n$.
Then the code $C-i$ (resp. $C/i$) obtained from~$C$ by \emph{puncturing} (resp. \emph{shortening}) at coordinate place~$i$. 
We denote by $C+i_{a}$ for $a \in \FF_{2}$ the subsets of $C$ where the
$i$-th entry of each codeword takes the value $a$ punctured at $i$.
We say $C-i$ \textit{unique} if no matter $i$-coordinate gives the same $C-i$.
A code is said to be $t-$\emph{homogeneous} if the codewords of every given
weight hold a~$t-$design. In particular, we call a code \emph{homogeneous}
instead of $t-$homogenous when $t =1$.

\begin{lem} \label{Lem:JWZ}
	Let $C$ be a code of length~$n$. 
	Then for every coordinate place~$i$, we have
	\begin{equation}\label{Equ:Target_1}
		J_{C,\{i\}}^{(g)}
		=
		\sum_{a = (a_{1},\ldots,a_{g}) \in \FF_{2}^{g}}
		y_{a}
		Z_{C+i_{a_{1}},\ldots,C+i_{a_{g}}},
	\end{equation}
	\begin{equation}\label{Equ:Target_2}
		W_{C-i}^{(g)} 
		= 
		\sum_{a = (a_{1},\ldots,a_{n})\in \FF_{2}^{g}}
		Z_{C+i_{a_{1}},\ldots,C+i_{a_{g}}},
	\end{equation}
	where 
	\[
	Z_{C+i_{a_{1}},\ldots,C+i_{a_{g}}}
	:=
	\sum_{u_{1}\in C+i_{a_{1}},\ldots,u_{g} \in C+i_{a_{g}}}
	\prod_{b\in \FF_{2}^{g}}
	x_{b}^{n_{b}(u_{1},\ldots,u_{g})}.
	\]
\end{lem}

\begin{proof}
	Statements $(\ref{Equ:Target_1})$ and $(\ref{Equ:Target_2})$
	can be shown immediately from the respective definitions.
\end{proof}

\begin{thm}\label{Thm:gthJacPolythomo}
	Let $C$ be a code of length~$n$.
	If $C$ is $t-$homogeneous,
	then the genus~$g$ Jacobi polynomial~$J_{C,T}^{(g)}$
	is independent in the choice of~$T \subseteq [n]$
	for~$|T| = t$.
\end{thm}

\begin{proof}
	Let $C$ be a code of length~$n$
	and 
	$C^{g} :=\underset{g \text{ times }}{\underbrace{C \times \cdots \times C}}$.
	We denote an element of $C^g$ by 
	\[
		\widetilde{c} := 
		(c_1,\ldots,c_n):=
			\begin{pmatrix}
				u_{11}&\ldots&u_{1n}\\
				u_{21}&\ldots&u_{2n}\\
				\vdots&\cdots&\vdots\\
				u_{g1}&\ldots&u_{gn}
			\end{pmatrix},
	\]
	where 
	$c_i:=(u_{1i},\ldots,u_{gi})^{t}\in \FF_{2}^{g}$. 
	Now for any $a \in \FF_{2}^{g}$
	and non-negative integer~$\ell$,
	we define
	\begin{align*}
		\supp_{a}(\widetilde{c})
		& :=
		\{
			i \in [n]
			\mid
			c_{i} = a
		\},\\
		C_{a,\ell}^{(g)}
		& :=
		\{
			\widetilde{c} \in C^{g}
			\mid
			n_{a}(\widetilde{c}) = \ell
		\},\\
		\mathcal{B}(C_{a,\ell}^{(g)})
		& :=
		\{
			\supp_{a}(\widetilde{c})
			\mid
			\widetilde{c}  \in C_{a,\ell}^{(g)}
		\},
	\end{align*}
	where, $n_{a}(\widetilde{c})$  is the number of 
	columns $i$ in $\widetilde{c}$ such that
	$a = c_{i}$.
	Since $C$ is $t$-homogeneous, therefore
	$\mathcal{B}(C_{a,\ell}^{(g)})$ forms a $t$-design
	for all given~$a \in \FF_{2}^{g}$ and $\ell$. 
	Hence, $C^{g}$ is also $t$-homogeneous.
	This conclude that $J_{C,T}^{(g)}$
	is independent in the choice of $T \subseteq [n]$
	for $|T|=t$.
\end{proof}

Now we introduce a higher genus generalization of 
the polarization operator given in~\cite{BMS}.
Let $f$ be a homogeneous polynomial in~$x_a$ 
and~$y_{a}$ for $a \in \mathbb{F}_2^g$.
We define the polarization operator $A_{(g)}$ in genus~$g$ 
for the homogeneous polynomial $f$ as
\[
	A_{(g)} . f 
	= 
	\left( 
	\sum_{a\in \mathbb{F}_2^g} 
	y_a 
	\frac{\partial }{\partial x_a}
	\right)f.
\]

Now we the following useful lemma.

\begin{lem} \label{Lem:1-homo}
	Let $C$ be a code of length~$n$. 
	If $C-i$ is unique for any $i$-th coordinate, then
	\[
		W_{C-i}^{(g)}
		= 
		\frac{1}{n} 
		\left( 
		\sum_{a \in \mathbb{F}_2^g} 
		\frac{\partial}{\partial x_a} 
		W_C^{(g)}
		\right).
	\]
\end{lem}

\begin{proof}
	Let $A({k_{a_1},\ldots,k_{a_{2^g}}})$ 
	where $a_{j} \in \mathbb{F}_2^g$ for all $j$ be the number of $c \in C^g$ such that
	
	\[
	n_{a_{j}}(c)=k_{a_{j}}.
	\]
	In other words, 
	$A({k_{a_1},\ldots,k_{a_{2^g}}})$
	is the coefficient of the monomial 
	$$x_{a_{1}}^{k_{a_{1}}}\cdots x_{a_{2^{g}}}^{k_{a_{2^{g}}}}.$$
	Let $B_{n}$ be the set of $2^{g}$-tuple 
	$(k_{a_{1}},\ldots,k_{a_{2^g}})$ in some fixed order
	such that 
	$a_{j} \in \FF_{2}^{g}$ for all~$j$ and
	$\sum_{j=1}^{2^{g}} k_{a_{j}} = n$.
	Then for any $a_{j} \in \FF_{2}^{g}$, we have
	
	\begin{align*}
		\sum_{i=1}^{n}
		Z&_{C+i_{a_{j1}},\ldots,C+i_{a_{jg}}}\\
		&= 
		\sum_{({k_{a_{1}},\ldots, k_{a_{j-1}},k_{a_{j}},k_{a_{j+1}},\ldots,k_{a_{2^g}}}) \in B_{n}}  
		k_{a_{j}} 
		A({k_{a_{1}},\ldots, k_{a_{j-1}},k_{a_{j}},k_{a_{j+1}},\ldots,k_{a_{2^g}}}) \\ 
		&\hspace{60pt}
		x_{a_{1}}^{k_{a_{1}}}
		\cdots
		x_{a_{j-1}}^{k_{a_{j-1}}}
		x_{a_j}^{k_{a_j} -1}
		x_{a_{j+1}}^{k_{a_{j+1}}}
		\cdots
		x_{a_{2^g}}^{k_{a_{2^g}}} \\
		& = 
		\frac{\partial }{\partial x_{a_{j}}} 
		W_C^{(g)}.
	\end{align*}
	
	
	Because no matter $i$ gives the same $C-i$, 
	therefore
	$Z_{C+i_{a_1},\ldots,C+i_{a_{g}}}$
	for $a = (a_{1},\ldots,a_{g}) \in \FF_{2}^{g}$
	can be obtained as follows
	\begin{align*}
		Z_{C+i_{a_{1}},\ldots,C+i_{a_{g}}}
		& = 
		\frac{1}{n}
		\frac{\partial }{\partial x_{a}} 
		W_C^{(g)}
	\end{align*}
	Then by Lemma~\ref{Lem:JWZ}, we have
	\[
		W_{C-i}^{(g)} 
		= 
		\frac{1}{n} 
		\sum_{a \in  \mathbb{F}_2^g} 
		\frac{\partial }{\partial x_a} 
		W_C^{(g)}.
	\]
	This completes the proof.
\end{proof}

In the situation that is discussed in Theorem~\ref{Thm:gthJacPolythomo}, 
we denote the genus $g$ Jacobi polynomial as $J_{C,t}^{(g)}$.
The following theorems are the generalization of \cite[Theorems 3, 4]{BMS}

\begin{thm}\label{Thm:1-homo_g}
If $C$ is $1-$homogeneous, then 
we have
\[
	J_{C,1}^{(g)} 
	= 
	\frac{1}{n} A_{(g)} W_C^{(g)}.
\]
\end{thm}

\begin{proof}
Apply Lemma~\ref{Lem:1-homo}. Then we get the relation. 
\end{proof}

Theorem~\ref{Thm:1-homo_g} has a meaning such the following. 
Let $W_C^{(g)}$ be the weight enumerator of a code $C$.
If $W_{C-i}^{(g)}$ is the same for all coordinate~$i$,
then $J_{C,i}$ can be obtained by averaging the summation of the product of $y_a$ and 
\[
	\frac{\partial}{\partial x_a} 
	W_C^{(g)}
\]
for $a \in \mathbb{F}_2^g$.
In other words, 
in this situation we can say that $J_{C,1}^{(g)}$ can be obtained by deleting coordinate $i$ on $[n]$ and moving it to another set which is still empty before $i$ is added.

\begin{thm} \label{Thm:t-homo_g}
Let $C$ be a $t-$homogeneous code. 
Then for all $T \subseteq [n]$ with $|T|=t$,
we have
\[
	J_{C,T}^{(g)} 
	= 
	\frac{1}{n(n-1)\ldots (n-t+1)} 
	A_{(g)}^t W_C^{(g)}.
\]
\end{thm}

\begin{proof}
We prove by induction. 
For $k<t$, we assume
\begin{equation*} \label{eq_prgenjac}
J_{C,k}^{(g)}= \frac{1}{n(n-1)\ldots (n-k+1)} A_{(g)}^k W_C^{(g)}.
\end{equation*}
By deleting a coordinate $i$ on $[n]\backslash [k]$ and move the coordinate $i$ to $[k]$,
we have that
\begin{align*}
J_{C,k+1} = & \frac{1}{n-k} \sum_{a \in \mathbb{F}_2^g} y_a \frac{\partial}{\partial x_a} J_{C,k} \\
= & \frac{1}{n-k} \left( A_{(g)} \frac{1}{n(n-1) \cdots (n-k+1)} A_{(g)}^k W_C^{(g)} \right)\\
= & \frac{1}{n(n-1) \cdots (n-k+1)(n-k)} A_{(g)}^{k+1} W_C^{(g)}.
\end{align*}
\end{proof}

\section{Split Jacobi polynomials}\label{Sec:SplitJacobiPoly}

In this section, we give the MacWilliams type identity for
the genus~$g$ Jacobi polynomial of a code attached to multiple sets 
of coordinate places.
Moreover, we give a split generalization of the polarization operation $A_{(g)}$, 
and obtain a split analogue of Theorems~\ref{Thm:1-homo_g} and \ref{Thm:t-homo_g}.

\begin{df}\label{Def:SplitCompWeight}
	Let $C$ be a code of length~$n$. Then 
	the genus~$g$ \emph{split weight enumerator} attached to $\ell$ mutually
	disjoint subsets $X_{1},\ldots,X_{\ell}$ of 
	coordinate places of the code $C$ such that 
	$$X_{1} \sqcup \cdots \sqcup X_{\ell} = [n]$$ is defined as follows:
	\[
	W_{C,X_{1},\ldots,X_{\ell}}^{(g)}
	(\{\{x_{i,a}\}_{a\in\FF_{2}^{g}}\}_{1\leq i \leq \ell})
	:=
	\sum_{u_{1},\ldots,u_{g} \in C}
	\prod_{i = 1}^{\ell}
	\prod_{a \in \FF_{2}^{g}}
	x_{i,a}^{n_{a,X_{i}}(u_{1},\ldots,u_{g})}.
	\]
	where $n_{a,X}(u_{1},\ldots,u_{g})$ 
	is the number of~$i$ such that $a = (u_{1i},\ldots,u_{gi}) \in \FF_{2}^{g}$ 
	on $X$.
\end{df}

\begin{df}\label{Def:SplitJacobi}
	Let $C$ be a code of length~$n$.
	Let $X_{1},\ldots, X_{\ell}$ be $\ell$ mutually
	disjoint sets such that 
	$$[n] = X_{1} \sqcup \cdots \sqcup X_{\ell}.$$ 
	Then the \emph{split Jacobi polynomial} of $C$ 
	attached to $T_{1},\ldots,T_{\ell}$ such that 
	$T_{i} \subseteq X_{i}$ for all $i$
	is defined by
	\begin{align*}
		J_{C,X_{1}(T_{1}),\ldots,X_{\ell}(T_{\ell})}
		&(\{w_{i},z_{i},x_{i}, y_{i}\}_{1\leq i \leq \ell})\\
		& := 
		\sum_{{u}\in C} 
		\prod_{i = 1}^{\ell}
		w_{i}^{m_{0,i}(u)}
		z_{i}^{m_{1,i}(u)}
		x_{i}^{n_{0,i}({u})}
		y_{i}^{n_{1,i}({u})},
	\end{align*}
	where $m_{a,i}(u)$ (resp. $n_{a,i} (u)$) 
	is the number of~$k$ such that $a = u_{k} \in \FF_{2}$ 
	on $T_{i}$ (resp. $X_{i} \backslash T_{i}$).
\end{df}

\begin{rem}
	If $\ell=1$, then the split Jacobi polynomial is the Jacobi polynomial 
	attached to a set~$T$ of coordinate places. 
\end{rem}

Now we give an arbitrary genus~$g$ generalization of split Jacobi polynomials
that we call $g$-th split Jacobi polynomials.
 
\begin{df}  \label{Def:GenSplitJacobi}
	Let $C$ be a code of length~$n$.
	Let $X_{1},\ldots, X_{\ell}$ be $\ell$ mutually
	disjoint sets such that 
	$$[n] = X_{1} \sqcup \cdots \sqcup X_{\ell}.$$ 
	Then the \emph{$g$-th split Jacobi polynomial} of $C$ 
	attached to $T_{1},\ldots,T_{\ell}$ such that 
	$T_{i} \subseteq X_{i}$ for all $i$
	is defined by
	\begin{multline*}
		J_{C,X_{1}(T_{1}),\ldots,X_{\ell}(T_{\ell})}^{(g)} 
		(\lbrace\lbrace y_{{i},a}, x_{{i},a} \rbrace_{a \in \mathbb{F}_2^g}\rbrace_{1 \leq i \leq \ell})\\
		:= 
		\sum_{u_1,...,u_g \in C } 
		\prod_{a \in \mathbb{F}_2^g} 
		\prod_{i = 1}^{\ell}
		y_{{i},a}^{m_{a,{i}} (u_{1},\ldots,u_{g})}\, 
		x_{{i},a}^{n_{a,i} (u_{1},\ldots,u_{g})}
	\end{multline*}
	where $m_{a,i}(u_{1},\ldots,u_{g})$ (resp. $n_{a,i} (u_{1},\ldots,u_{g})$) 
	is the number of~$k$ such that $a = (u_{1k},\ldots,u_{gk}) \in \FF_{2}^{g}$ 
	on $T_{i}$ (resp. $X_{i} \backslash T_{i}$).
\end{df}

\begin{rem}
	If $\ell=1$, then the $g$-th split Jacobi polynomial is 
	the $g$-th Jacobi polynomial attached to a set~$T$ of coordinate places. 
\end{rem}

The $g$-th split Jacobi polynomial of a code 
attached to multiple disjoint sets of coordinate places 
satisfies the following MacWilliams type identity.

\begin{thm}\label{Thm:gthSplitJacobiMacWilliams}
	Let $C$ be a code of length $n$.
	Let $X_{1},\ldots, X_{\ell}$ be $\ell$ mutually
	disjoint sets such that 
	$[n] = X_{1} \sqcup \cdots \sqcup X_{\ell}.$
	Then
	\begin{align*}
		J&_{C^{\perp},X_{1}(T_{1}),\ldots,X_{\ell}(T_{\ell})}^{(g)} 
		(\{\{y_{i,a},x_{i,a}\}_{a \in \FF_{2}^{g}}\}_{1\leq i\leq \ell})\\
		& = 
		\dfrac{1}{|C|^{g}} 
		J_{C,X_{1}(T_{1}),\ldots,X_{\ell}(T_{\ell})}^{(g)}
		\left(
		\left\{
		\left\{
		\sum_{b \in \FF_{2}^{g}}
		(-1)^{a \cdot b}
		y_{i,b},
		\sum_{b \in \FF_{2}^{g}}
		(-1)^{a \cdot b}
		x_{i,b}
		\right\}_{a \in \FF_{2}^{g}}
		\right\}_{1\leq i \leq \ell}
		\right).
	\end{align*}
\end{thm}

\begin{proof}
	Let $C$ be a code of length~$n$. 
	For ${v} \in \FF_{2}^{n}$, define
	\[
	\delta_{C^{\perp}}({v}) 
	:= 
	\begin{cases}
		1 & \mbox{if } {v} \in C^{\perp}, \\
		0 & \mbox{otherwise}.
	\end{cases} 
	\]
	Then we have the following identity:
	\[
		\delta_{C^\perp}({v})
		= 
		\dfrac{1}{|C|} 
		\sum_{{u} \in C} 
		(-1)^{{u} \cdot {v}}.
	\]
	
	\begin{align*}
		J&_{C^{\perp},X_{1}(T_{1}),\ldots,X_{\ell}(T_{\ell})}^{(g)} 
		(\{\{y_{{i},a},x_{{i},a}\}_{a \in \FF_{2}^{g}}\}_{1\leq i\leq \ell})\\
		& = 
		\sum_{{u}_{1},\ldots,u_{g}\in C^{\perp}} 
		\prod_{i = 1}^{\ell}
		\prod_{a \in \FF_{2}^{g}}
		y_{{i},a}^{m_{a,{i}}({u}_{1},\ldots,u_{g})}
		x_{{i},a}^{n_{a,i}({u}_{1},\ldots,u_{g})}\\
		& = 
		\sum_{{v}_{1},\ldots,v_{g} \in \FF_{2}^{n}} 
		\delta_{C^\perp}({v}_{1}) 
		\cdots
		\delta_{C^\perp}({v}_{g})
		\prod_{i = 1}^{\ell}
		\prod_{a \in \FF_{2}^{g}}
		y_{{i},a}^{m_{a,i}({v}_{1},\ldots,v_{g})}
		x_{{i},a}^{n_{a,i}({v}_{1},\ldots,v_{g})}\\
		& =
		\dfrac{1}{|C|^{g}} 
		\sum_{\substack{{u}_{1},\ldots, u_{g} \in C\\ 
				{v}_{1},\ldots, v_{g}\in \FF_{2}^{n}}}  
		(-1)^{{u}_{1}\cdot {v}_{1}+\cdots+u_{g}\cdot v_{g}} 
		\prod_{i = 1}^{\ell}
		\prod_{a \in \FF_{2}^{g}}
		y_{{i},a}^{m_{a,{i}}({v}_{1},\ldots,v_{g})}
		x_{{i},a}^{n_{a,{i}}({v}_{1},\ldots,v_{g})}\\
		& = 
		\dfrac{1}{|C|^{g}}
		\sum_{\substack{{u}_{1},\ldots,u_{g} \in C\\ 
				{v}_{1},\ldots,v_{g} \in \FF_{2}^{n}}} 
		(-1)^{\sum_{k = 1}^{g} u_{k1}v_{k1} + \cdots + u_{kn}v_{kn}}
		\prod_{i = 1}^{\ell}
		\left(
		\prod_{j \in T_{i}}
		y_{{i},v_{1j}\cdots v_{gj}}
		\right)
		\left(
		\prod_{j \in X_{i}\setminus T_{i}}
		x_{{i},v_{1j} \cdots v_{gj}}
		\right)\\
		& = 
		\dfrac{1}{|C|^{g}}
		\sum_{{u}_{1},\ldots, u_{g} \in C}
		\prod_{i = 1}^{\ell}
		\left(
		\prod_{j \in T_{i}} 
		\sum_{v_{1j},\ldots, v_{gj} \in \FF_{2}} 
		(-1)^{u_{1j}v_{1j}+\cdots + u_{gj}v_{gj}} 
		y_{{i},v_{1j}\cdots v_{gj}} 
		\right)\\
		&\hspace{1in}
		\left(
		\prod_{j \in X_{i}\setminus T_{i}} 
		\sum_{v_{1j},\ldots, v_{gj} \in \FF_{2}} 
		(-1)^{u_{1j}v_{1j}+\cdots + u_{gj}v_{gj}} 
		x_{{i},v_{1j}\cdots v_{gj}} 
		\right)\\
		& = 
		\dfrac{1}{|C|^{g}}
		\sum_{{u}_{1},\ldots,u_{g} \in C}
		\prod_{i=1}^{\ell}
		\prod_{a\in\FF_{2}^{g}} 
		\left(
		\sum_{b \in \FF_{2}^{g}}
		(-1)^{a \cdot b} 
		y_{{i},b}
		\right)
		^{m_{a,{i}}({u}_{1},\ldots,u_{g})}
		\left(
		\sum_{b \in \FF_{2}^{g}}
		(-1)^{a \cdot b} 
		x_{{i},b}
		\right)
		^{n_{a,{i}}({u}_{1},\ldots, u_{g})}  \\
		& = 
		\dfrac{1}{|C|^{g}}
		J_{C,X_{1}(T_{1}),\ldots,X_{\ell}(T_{\ell})}^{(g)}
		\left(
		\left\{
		\left\{
		\sum_{b \in \FF_{2}^{g}}
		(-1)^{a \cdot b} 
		y_{{i},b},
		\sum_{b \in \FF_{2}^{g}}
		(-1)^{a \cdot b} 
		x_{{i},b}
		\right\}_{a \in \FF_{2}^{g}}
		\right\}_{1\leq i \leq \ell}
		\right).	
	\end{align*}
	Hence the proof is completed.
\end{proof}

The following result reflects the basic motivation to introduce 
the concept of split complete Jacobi polynomials attached to multiple sets. 
We omit the proof of the theorem since it is a straightforward generalization
of Theorem~\ref{Thm:gthJacPolythomo}. 

\begin{thm}\label{Thm:JacToDesign}
	Let $C$ be a code of length~$n$. 
	Let $\bm{v} := (v_{1},\ldots,v_{\ell})$ such that $\sum_{i}^{\ell}v_{i} = n$.
	Let $\bm{X} := (X_{1},\ldots,X_{\ell})$ of pairwise disjoint set 
	$X_{i}\subseteq [n]$ 
	with $|X_{i}| = v_{i}$ for all $i$.
	If $C$ is $\ell$-th $t$-homogenous
	with $\bm{t} = (t_{1},\ldots,t_{\ell})$ such that 
	$\sum_{i = 1}^{\ell}t_{i} = t$
	then
	the genus~$g$ split Jacobi polynomial 
	$J_{C,X_{1}(T_{1}),\ldots,X_{\ell}(T_{\ell})}^{(g)}$ 
	with $T_{i} \in \binom{X_{i}}{t_{i}}$
	for all $i$ is independent of the choices of the sets $T_{1},\ldots,T_{\ell}$.
\end{thm}

In the situation described in the above theorem, 
$J_{C,X_{1}(T_{1}),\ldots,X_{\ell}(T_{\ell})}^{(g)}$
the split Jacobi polynomials in genus~$g$ 
is independent of the choices of the sets $T_{1},\ldots,T_{\ell}$. 
In this case, we prefer to denote the split Jacobi polynomials
in genus~$g$ as 
$J_{X_{1}(t_1),\ldots,X_{\ell}(t_{\ell})}^{(g)}$. 
In particular, when $\bm{k} = (k)$ and $\bm{t} = (t)$, 
the split Jacobi polynomials in genus~$g$ are coincide with
the Jacobi polynomials in genus~$g$.

Let $\ell$, $n$ be the positive integers such that $\ell \leq n$. 
Let $\bm{v}:= (v_{1},\ldots,v_{\ell})$ 
such that $\sum_{i=1}^{\ell} v_{i} = n$. 
Let $P(\{\{y_{{i},a},x_{{i},a}\}_{a \in \FF_{2}^{g}}\}_{1\leq i \leq \ell})$ be a polynomial of degree~$n$ in $2^{g+1}\ell$ variables 
such that in its each term the sum of the powers of 
$x_{{i},a}$ and $y_{{i},a}$ for all 
$a\in\FF_{2}^{g}$ is~$v_{i}$ for all~$i$. 
Define the polarization operator 
$A_{(g),\ell}(k)$ 
for any integer $1 \leq k \leq \ell$
as follows:
\begin{equation*}\label{Equ:Operator}
	A_{(g),\ell}(k)
	\cdot
	P
	:= 
	\frac{1}{v_{k}}
	\left(
	\sum_{a\in\FF_{2}^{g}}
	y_{{k},a} 
	\frac{\partial}{\partial x_{{k},a}}
	\right)
	P.
\end{equation*}
If $\ell = 1$, 
then $A_{(g),\ell}(k)$ coincide with the polarization operator as $A_{(g)}$
that discussed in Section~\ref{Sec:JacobiPoly}. 

Now we have the following split analogues of Theorems~\ref{Thm:1-homo_g} and \ref{Thm:t-homo_g}.

\begin{thm}\label{Thm:Main1}
	Let $C$ be a code of length $n$. 
	Let $\bm{v} := (v_{1},\ldots,v_{\ell})$ 
	such that $\sum_{i=1}^{\ell} v_{i} = n$.
	We also let $X_{1},\ldots,X_{\ell}$
	be the mutually disjoint subsets of $[n]$ such that $X_{1} \sqcup\cdots\sqcup X_{\ell} = [n]$ and $|X_{i}| = v_{i}$ for all $i$. 
	If~$C$ is $\ell$-th $1-$homogeneous, 
	then 
	\[
		J_{C,X_{1}(\emptyset),\ldots,X_{k-1}(\emptyset),X_{k}(\{i\}),X_{k+1}(\emptyset),\ldots,X_{\ell}(\emptyset)}^{(g)}
		=
		A_{(g),\ell}(k)
		\cdot
		W_{C,X_{1},\ldots,X_{k-1},X_{k},X_{k+1},\ldots,X_{\ell}}^{(g)}.
	\]
\end{thm}

\begin{proof}
	Following the similar arguments given in Theorem~\ref{Thm:1-homo_g},
	we can have the relation.
\end{proof}

\begin{thm}\label{Thm:SplitMain2}
	Let $C$ be a code of length $n$. 
	Let $\bm{v} := (v_{1},\ldots,v_{\ell})$ 
	such that $\sum_{i=1}^{\ell} v_{i} = n$.
	We also let $X_{1},\ldots,X_{\ell}$
	be the mutually disjoint subsets of $[n]$ such that $X_{1} \sqcup\cdots\sqcup X_{\ell} = [n]$ and $|X_{k}| = v_{k}$ for all $k$.
	If~$C$ is $\ell$-th $t$-homogeneous and contains no codeword of
	Hamming weight less than $t$, 
	then for $\bm{t} := (t_{1},\ldots,t_{\ell})$ 
	such that $\sum_{i=1}^{\ell}t_{i} = t$, 
	we have
	\[
		J_{C, X_{1}(T_1),\ldots,X_{\ell}(T_{\ell})}^{(g)}
		=
		A_{(g),\ell}^{t_{\ell}}(\ell)\cdots A_{(g),\ell}^{t_{1}}(1)
		\cdot
		W_{C,X_{1},\ldots,X_{\ell}}^{(g)},
	\] 
	for each $(T_{1},\ldots,T_{\ell}) \in \binom{X_{1}}{t_{1}} \times \cdots \times \binom{X_{\ell}}{t_{\ell}}$.
\end{thm}

\section{Invariant Rings}\label{Sec:InvRing}

We refer the readers to~\cite{BannaiOzekiTeranishi, NRS} 
for detailed discussions on the invariant rings.
Here we give a brief discussion
that relevant to some useful facts and notations 
from invariant theory.
Let $G$ be a finite $m\times m$ matrix group that acts 
on a polynomial ring $\mathfrak{R}:=\CC[x_0,\ldots,x_{m-1}]$; 
for $\sigma\in G$ and $f(x_0,\ldots,x_{m-1})\in \mathfrak{R}$, 
\[
\sigma f(x_0,\ldots,x_{m-1})
=
f(\sigma(x_0,\ldots,x_{m-1})^t). 
\]
Then 
$G$ acts on the polynomial ring 
$\mathfrak{R}$
in a natural way. 
Set
\[
\mathfrak{R}^{G}
:=
\{
f \in 
\mathfrak{R}
\mid
\sigma f = f
\mbox{ for all }
\sigma \in G
\}.
\]
Then the dimension formula for the invariant ring~$\mathfrak{R}^{G}$ is
\[
\Phi_{G}(t)
=\sum_{d \ge 0}
(\dim\mathfrak{R}^{G}_{d}) t^{d},
\]
where $\mathfrak{R}^{G}_{d}$ is the homogeneous part of degree~$d$ in~$\mathfrak{R}^{G}$.
Therefore, the ring~$\mathfrak{R}^{G}$ is graded as:
\[
\mathfrak{R}^{G}
=
\bigoplus_{d=0}^{\infty}
\mathfrak{R}_{d}^{G}.
\]
Furthermore, Molien~\cite{Molien} presented a beautiful expression 
of $\Phi_{G}(t)$ as: 
\[
\Phi_{G}(t)
=
\frac{1}{|G|}
\sum_{\sigma\in G}
\frac{1}{\det{(1-t\sigma)}}.
\]

In particular, let $H$ be the group generated by
\[
\sigma_{1} = \frac{1}{\sqrt{2}}
\begin{pmatrix*}[r]
	1 & 1\\
	1 & -1
\end{pmatrix*}
\mbox{ and }
\sigma_2
=
\begin{pmatrix}
	1 & 0\\
	0 & i
\end{pmatrix}.
\]
This group is of order~$192$ and is known 
as finite unitary reflection group No.~$9$ 
in Shephard and Todd's list~\cite{ST}.
A finite group whose invariant ring is generated 
by the algebraically independent elements over~$\CC$ 
is called a \emph{finite unitary refection group}.
Here $\CC[x_{0},x_{1}]^{H}$ is the invariant ring under the
action of group~$H$.
Gleason~\cite{Gleason} pointed out a remarkable fact 
that the invariant ring 
$\CC[x_{0},x_{1}]^{H}$ is generated by two 
weight enumerators 
$W_{d_{8}^{+}}(x_{0},x_{1})$ and $W_{g_{24}}(x_{0},x_{1})$ 
of Type~$\II$ codes~$d_{8}^{+}$ and $g_{24}$, 
where 
the code $d_n^+$ means the code with the generator matrix:
\[\left(
\begin{array}{*{11}c}
	1 & 1 & 1 & 1 & 0 & 0 &  \ldots & 0 & 0 & 0 & 0 \\
	0 & 0 & 1 & 1 & 1 & 1 &  \ldots & 0 & 0 & 0 & 0 \\
	\vdots & \vdots & \ddots & &&  & \ddots & & & \\
	0 & 0 & 0 & 0 & 0 & 0 &  \ldots & 1 & 1 & 1 & 1 \\
	1 & 0 & 1 & 0 & 1 & 0 &  \ldots & 1 & 0 & 1 & 0
\end{array} \right)
\]
and $g_{24}$ is an extremal Golay code of length~$24$.
However, a new observation investigated by Fujii and Oura~\cite{FO2018}
showed that
\[
	\CC[x_{0},x_{1}]^{H}
	=
	\CC[W_{d_{8}^{+}}(x_{0},x_{1}), W_{d_{24}^{+}}(x_{0},x_{1})],
\]
where 
$W_{d_{8}^{+}}(x_{0},x_{1})$ and $W_{d_{24}^{+}}(x_{0},x_{1})$
are algebraically independent over~$\CC$.
In this section, 
we shall construct the invariant rings 
under the action of finite unitary  
reflection groups
and obtain some similar equalities as above
for Jacobi polynomials for arbitrary genus~$g$.  

Now let us define the following matrices in $\mathrm{GL}(2^g,\CC)$:

\begin{align*}
	T_{g}
	& :=
	\left( 
	\frac{1+i}{2} 
	\right)^g 
	\left( 
	(-1)^{a \cdot b}
	\right)_{a,b \in \mathbb{F}_2^{g}},\\
	E_{g}
	& :=
	\diag(1,i,1,i,\dots).
\end{align*}
Then
\[
	\widetilde{T}_{g}
	:=
	\begin{pmatrix}
		T_{g} & \\
		& T_{g}
	\end{pmatrix}
	\mbox{ and }
	\widetilde{E}_{g}
	:=
	\begin{pmatrix}
		E_{g} & \\
		& E_{g}
	\end{pmatrix}.
\] 

Let $C$ be a binary code of length~$n$.
Then by using the 
transformation~$\widetilde{T}_{g}$,
the MacWilliams identity for 
Jacobi polynomials of~$C$ in genus~$g$
stated in~Theorem~\ref{Thm:JacMacGen_g}
can be restated as:
\begin{multline*}
	\widetilde{T}_{g}
	\left(
	\left(\frac{1}{\sqrt{2}}\right)^{g(\dim C)}
	J_{C,T}^{(g)}
	(\{\{y_{a},x_{a}\}_{a \in \FF_{2}^{g}}\})
	\right)\\
	=
	(\eta_{8}^{ng})
	\left(\frac{1}{\sqrt{2}}\right)^{g(\dim C^{\perp})}
	J_{C^{\perp},T}^{(g)}
	(\{\{y_{a},x_{a}\}_{a \in \FF_{2}^{g}}\}).
\end{multline*}

From the above identity and definitions of~$\widetilde{T}_{g}$
and~$\widetilde{E}_{g}$, we can conclude the following remarks.
 
\begin{rem}\label{Rem:JacTg}
	The Jacobi polynomials in genus~$g$ 
	of a self-dual code
	is remain invariant under the transformation~$\widetilde{T}_{g}$.
\end{rem}

\begin{rem}\label{Rem:JacEg}
	The Jacobi polynomials in genus~$g$ 
	of a double even code
	is remain invariant under the transformation~$\widetilde{E}_{g}$.
\end{rem}

Let $\sigma \in \FF_2^{g} \rtimes \mathrm{GL}(g,\mathbb{F}_2)$ 
together with the natural action on~$\FF_{2}^{g}$ 
defined by~$\sigma(a) := (v,M)(a) := Ma + v$. 
Now we construct a $(0,1)-$matrix
$M_g(\sigma)$, 
where $M_g(\sigma)$ has non-zero entries on $(a,\sigma(a))$ 
for $a \in \mathbb{F}_2^g$.
Then
\[
	\widetilde{M}_{g}(\sigma)
	:=
	\begin{pmatrix}
		M_{g}(\sigma) & \\
		& M_{g}(\sigma)
	\end{pmatrix}.
\]
Therefore, 
the matrix $\widetilde{M}_g (\sigma)$ 
transform $\lbrace x_a,y_a \rbrace_{a \in \mathbb{F}_2^g}$ to $\lbrace \sigma(x_a),\sigma(y_a) \rbrace_{a \in \FF_{2}^{g}}$.
This implies the following relation. 
\begin{align*}
	\widetilde{M}_{g}(\sigma)
	\left(J_{C,T}^{(g)} (\lbrace x_a, y_a \rbrace_{a \in \mathbb{F}_2^g})\right)
	= J_{C,T}^{(g)} (\lbrace x_{\sigma (a) }, y_{\sigma (a)} \rbrace_{a \in \mathbb{F}_2^g}).
\end{align*}

From the above relation, we can have the following fact.

\begin{rem}\label{Rem:JacMgsigma}
	The Jacobi polynomials in genus~$g$ 
	of a binary code containing all one vector
	is remain invariant under the transformation~$\widetilde{M}_{g}(\sigma)$.
\end{rem}

Let 
\[
	G_{g}
	:=
	\langle
		\widetilde{T}_{g},
		\widetilde{E}_{g},
		\widetilde{M}_{g}(\sigma),
		\eta_{8}
		\mid
		\mbox{ for all }
		\sigma \in \FF_{2}^{g} \rtimes \mathrm{GL}(g,\FF_{2})
	\rangle
\]
be the subgroup of $\mathrm{GL}(2^{g+1},\CC)$,
where $\eta_{8} := (1+i)/\sqrt{2}$ 
is the primitive $8$th root of unity. 
The orders of group~${G}_{g}$ 
for~$g = 1, 2$
are shown in Table~\ref{Tab:OrderGgell}.

The above discussions and Remarks~\ref{Rem:JacTg}, \ref{Rem:JacEg} and~\ref{Rem:JacMgsigma} proved the following result.

\begin{thm}\label{Thm:JacInvUnderGg}
	The Jacobi polynomials in genus~$g$ of a Type~$\II$ code
	is invariant under the action of~$G_{g}$.
\end{thm}

Let $\mathfrak{J}^{(g)}$ be the ring generated over~$\CC$ 
by the Jacobi polynomials in genus~$g$ of Type~$\II$ codes. 
Then this ring is graded as
\[
	\mathfrak{J}^{(g)}
	=
	\bigoplus_{n \equiv 0 \pmod 8}
	\mathfrak{J}_{n}^{(g)}.
\]
By Theorem~\ref{Thm:JacInvUnderGg}, 
it is immediate that~$\mathfrak{J}^{(g)}$
is an invariant ring under
the action of group~${G}_{g}$ in a natural way.
Therefore, we have the dimension formula for the invariant 
ring~$\mathfrak{J}^{(g)}$ 
as follows:
\begin{align*}
	\Phi_{{G}_{g}}(t)
	& :=
	\sum_{n \equiv 0\pmod 8}
	\left(\dim\mathfrak{J}_{n}^{(g)}\right) 
	t^{n}\\
	& =
	\frac{1}{|{G}_{g}|}
	\sum_{\sigma\in {G}_{g}}
	\frac{1}{\det{(1-t\sigma)}},
\end{align*}
where $\mathfrak{J}_{n}^{(g)}$ 
is the homogeneous part of degree~$n$
in the ring $\mathfrak{J}^{(g)}$.

\begin{thm}\label{Thm:MolienJacG11}
	The dimension formula for the ring $\mathfrak{J}^{(1)}$
	is as follows:
	\begin{align*}
		\Phi_{{G}_{1}}(t)
		& = 
		\sum_{n \equiv 0\pmod 8}
		\left(\dim\mathfrak{J}_{n}^{(1)}\right) 
		t^{n}\\
		& =
		\frac{1+8 t^8 + 21 t^{16} + 58 t^{24} +47 t^{32}+ 35 t^{40}+ 21 t^{48} 	+ t^{56}}{(1-t^8)^2 \, (1-t^{24})^2}\\
		& = 
		1 + 10 t^{8} + 40 t^{16} + 130 t^{24} + 283 t^{32} + 513 t^{40} + \cdots.
	\end{align*}
\end{thm}

\begin{thm}\label{Thm:MolienJac21}
	The dimension formula for the ring $\mathfrak{J}^{(2)}$
	is as follows:
	\begin{align*}
		\Phi_{{G}_{2}}(t) 
		= &
		\sum_{n \equiv 0\pmod 8}
		\left(\dim\mathfrak{J}_{n}^{(2)}\right) 
		t^{n} \\
		= & 
		\, \, (t^{184} + 8t^{176} + 49t^{168} + 325t^{160} + 1240t^{152} + 3421t^{144} \\
		&+ 7987t^{136} + 15287t^{128} + 24892t^{120} + 35648t^{112} + 45097t^{104} \\
		& + 50365t^{96}  + 50365t^{88} + 45097t^{80} + 35648t^{72} + 24892t^{64} \\
		& + 15287t^{56} + 7987t^{48} + 3421t^{40} + 1240t^{32} + 325t^{24} + 49t^{16}\\
		& + 8t^{8} + 1) / ((1-t^8)^2 (1-t^{24})^4 (1-t^{40})^2) \\
		= & \, \, 1 + 10 t^{8} + 68 t^{16} + 455 t^{24} + 2114 t^{32} + 7392 t^{40} + \cdots.
	\end{align*}
\end{thm}

\begin{table}[t]
	\centering
	\begin{tabular}{c | cc}
		$g$ & 1  & 2 \\
		\hline
		${G}_{g}$ & 192 & 92160 \\
	\end{tabular} 
	\caption{Order of ${G}_{g}$}
	\label{Tab:OrderGgell}
\end{table}

\section{Ring of Jacobi polynomials}\label{Sec:GenInvRing}

In this section, 
we give the computation results for $g=1,2$
that obtain the bases of $\mathfrak{J}_{n}^{(g)}$
for $n \equiv 0 \pmod 8$. 
We show the dimension up to $|T|=n/2$. 
From Remark~\ref{rem:relation_on_Jacobi},
the dimension of $\mathfrak{J}_n^{(g)}$ with $|T|>n/2$ is equal to the dimension of $\mathfrak{J}_n^{(g)}$ with $|T|= |[n] \backslash T|$.

\subsection{In genus~$1$}

It is immediate from Theorem~\ref{Thm:MolienJacG11} that
$\dim \mathfrak{J}_{8}^{(1)} = 10$. 
We have the basis elements by computing $J_{d_{8}^{+},T}^{(1)}$ 
for $|T|=0,1,\cdots,8$.
For $|T| \neq 4$, 
we compute $8$ independent Jacobi polynomials 
$J_{d_{8}^{+},T}^{(1)}$.
The Jacobi polynomials $J_{d_{8}^{+},T}^{(1)}$ with $|T|=4$ is not unique. 
There are~$2$ independent $J_{d_{8}^{+},T}^{(1)}$ 
with respect to~$T$ where~$|T| = 4$.  
In particular, the choices of $T$'s are: 
$T_{1}=[1,2,3,4]$ and $T_{2}=[1,2,3,5]$.
Immediately,
these $10$ independent Jacobi polynomials 
generates~$\mathfrak{J}_{8}^{(1)}$ 
(see Table~\ref{tab_JandR}).
The basis element of $\mathfrak{J}_{8}$ are the following. 
\begin{align*}
	J_{C,0} & = x_{0}^{8} + 14 x_{0}^{4} x_{1}^{4} + x_{1}^{8}, \\
	J_{C,1} & = x_{0}^{7} y_{0} + 7 x_{0}^{3} x_{1}^{4} y_{0} + 7 x_{0}^{4} x_{1}^{3} y_{1} + x_{1}^{7} y_{1}, \\
	J_{C,2} & = x_{0}^{6} y_{0}^{2} + 3 x_{0}^{2} x_{1}^{4} y_{0}^{2} + 8 x_{0}^{3} x_{1}^{3} y_{0} y_{1} + 3 x_{0}^{4} x_{1}^{2} y_{1}^{2} + x_{1}^{6} y_{1}^{2}, \\
	J_{C,3} & = x_{0}^{5} y_{0}^{3} + x_{0} x_{1}^{4} y_{0}^{3} + 6 x_{0}^{2} x_{1}^{3} y_{0}^{2} y_{1} + 6 x_{0}^{3} x_{1}^{2} y_{0} y_{1}^{2} + x_{0}^{4} x_{1} y_{1}^{3} + x_{1}^{5} y_{1}^{3}, \\
	J_{C,T_1} & = x_{0}^{4} y_{0}^{4} + x_{1}^{4} y_{0}^{4} + 12 x_{0}^{2} x_{1}^{2} y_{0}^{2} y_{1}^{2} + x_{0}^{4} y_{1}^{4} + x_{1}^{4} y_{1}^{4}, \\
	J_{C,T_2} & = x_{0}^{4} y_{0}^{4} + 4 x_{0} x_{1}^{3} y_{0}^{3} y_{1} + 6 x_{0}^{2} x_{1}^{2} y_{0}^{2} y_{1}^{2} + 4 x_{0}^{3} x_{1} y_{0} z^{3} + x_{1}^{4} y_{1}^{4}, \\
	J_{C,5} & = x_{0}^{3} y_{0}^{5} + x_{1}^{3} y_{0}^{4} y_{1} + 6 x_{0} x_{1}^{2} y_{0}^{3} y_{1}^{2} + 6 x_{0}^{2} x_{1} y_{0}^{2} y_{1}^{3} + x_{0}^{3} y_{0} y_{1}^{4} + x_{1}^{3} y_{1}^{5}, \\
	J_{C,6} & = x_{0}^{2} y_{0}^{6} + 3 x_{1}^{2} y_{0}^{4} y_{1}^{2} + 8 x_{0} x_{1} y_{0}^{3} y_{1}^{3} + 3 x_{0}^{2} y_{0}^{2} y_{1}^{4} + x_{1}^{2} y_{1}^{6}, \\
	J_{C,7} & = x_{0} y_{0}^{7} + 7 x_{1} y_{0}^{4} y_{1}^{3} + 7 x_{0} y_{0}^{3} y_{1}^{4} + x_{1} y_{1}^{7}, \\
	J_{C,8} & = y_{0}^{8} + 14 y_{0}^{4} y_{1}^{4} + y_{1}^{8}.
\end{align*}
Again $\dim \mathfrak{J}_{16}^{(1)} = 40$. 
Our computation shows that 
$\mathfrak{J}_{8}^{(1)} \mathfrak{J}_{8}^{(1)}$ 
is enough to generate 
$\mathfrak{J}_{16}^{(1)}$
(see Table~\ref{tab_JandR}).
Moreover,  $\dim \mathfrak{J}_{24}^{(1)} = 130$.
We obtain $105$ basis elements of $\mathfrak{J}_{24}^{(1)}$ 
by $\mathfrak{J}_8^{(1)}\mathfrak{J}_8^{(1)}\mathfrak{J}_8^{(1)}$.
The remaining $25$ basis elements are taken by computing 
$J_{d_{24}^{+},T}^{(1)}$ with respect to~$T$ 
for~$|T|=0,1, \cdots, 24$.
The detail choices of $T$ is shown in Table \ref{tab_T24}, we omit the cases $|T|\leq 5$ and $|T|\geq 19$ because of 
the uniqueness of the Jacobi polynomials $J_{d_{24}^{+},T}^{(1)}$.

\begin{table}[ht!]
	\centering
	\begin{tabular}{cc}
		$|T|$ & $T$ \\
		\hline
		6  & $\left[1, 2, 3, 4, 5, 6\right]$ \\
		7  & $\left[1, 3, 4, 5, 6, 7, 8\right]$\\
		8  & $\left[1, 2, 3, 4, 5, 6, 7, 8\right]$\\
		9  & $\left[1, 3, 4, 5, 6, 7, 8, 9, 10\right]$\\
		10 & $\left[1, 2, 3, 4, 5, 6, 7, 8, 9, 10\right]$\\
		11 & $\left[1, 3, 4, 5, 6, 7, 8, 9, 10, 11, 12\right]$\\
		12 & $\left[1, 2, 3, 4, 5, 6, 7, 8, 9, 10, 11, 12\right]$\\
		13 & $\left[1, 3, 4, 5, 6, 7, 8, 9, 10, 11, 12, 13, 14\right]$\\
		14 & $\left[1, 2, 3, 4, 5, 6, 7, 8, 9, 10, 11, 12, 13, 14\right]$\\
		15 & $\left[1, 3, 4, 5, 6, 7, 8, 9, 10, 11, 12, 13, 14, 15, 16\right]$\\
		16 & $\left[1, 2, 3, 4, 5, 6, 7, 8, 9, 10, 11, 12, 13, 14, 15, 16\right]$\\
		17 & $\left[1, 3, 4, 5, 6, 7, 8, 9, 10, 11, 12, 13, 14, 15, 16, 17, 18\right]$\\
		18 & $\left[1, 2, 3, 4, 5, 6, 7, 8, 9, 10, 11, 12, 13, 14, 15, 16, 17, 18\right]$
	\end{tabular}
	\caption{Set of coordinates $T$}
	\label{tab_T24}
\end{table}

The computation shows that the basis of $\mathfrak{J}_8^{(1)}$ and $\mathfrak{J}_{24}^{(1)}$ are enough to generate 
$\mathfrak{J}_{n}^{(1)}$ for $n>24$
(see Table~\ref{tab_JandR}).

\begin{table}[h] 
	\centering
	\begin{tabular}{c|c|c|c|c|c|c|c}
		$n$ & 8 & 16 & 24 & 32 & 40 & 48 & 56 \\
		\hline
		$\dim \, \mathfrak{J}_{n}^{(1)}$ & 10 & 40 & 130 & 283 & 513 & 883 & 1372 \\
		$\dim \, \CC[\mathfrak{J}_8^{(1)}]_{n} $ & 10 & 40 & 105 & 219 & 396 & 650 & 995 \\
		$\dim \, \CC[\mathfrak{J}_8^{(1)}, \mathfrak{J}_{24}^{(1)}]_{n}$ & 10 & 40 & 130 & 283 & 513 & 883 & 1372 
	\end{tabular}
	\caption{Comparisons of dimensions}
	\label{tab_JandR}
\end{table} 

From the above discussions and the comparisons of dimensions 
that are shown in Table~\ref{tab_JandR},
we have the following result.

\begin{thm} \label{Thm:JacGenJ8J24}
	The invariant ring $\mathfrak{J}^{(1)}$ 
	is generated by the bases of 
	$\mathfrak{J}_8^{(1)}$ and $\mathfrak{J}_{24}^{(1)}$.
\end{thm}

From the dimension formula of the invariant ring~$\mathfrak{J}^{(1)}$
given in Theorem~\ref{Thm:MolienJacG11}, 
we can find that there are~$4$ algebraically independent generators. 
Based on our computation, 
the generators which are algebraically independent are 
\[
J_{d_8^+, \emptyset}^{(1)},
J_{d_8^+, [8]}^{(1)},
J_{d_{24}^+, \emptyset}^{(1)},
J_{d_{24}^+, [24]}^{(1)}.
\]

\subsection{In genus~$2$}

In this part, we show our results up to $n=24$.
For $n=8,16$, the code $d_n^+$ is enough to generate 
$\mathfrak{J}_{n}^{(2)}$.
However, we use the codes $d_{24}^+$ and $g_{24}$ for $n=24$.

For $n=8$,
the dimension of $\mathfrak{J}_8^{(2)}$ for each $T$ is shown in Table \ref{tab_g2n8}.
The sets of coordinates used are the same as $g=1$ case. 

\begin{table}
\centering
\begin{tabular}{cc}
$|T|$ & $\mathfrak{J}_8^{(2)}$ \\
\hline
0 & 1 \\
1 & 1 \\
2 & 1 \\
3 & 1 \\
4 & 2
\end{tabular}
\caption{Number of basis elements in $\mathfrak{J}_{8}^{(2)}$ attached to $T$}
\label{tab_g2n8}
\end{table}

We continue for $n=16$.
From the computation,
the multiplication of two elements of $\mathfrak{J}_8^{(2)}$ is not enough to generate $\mathfrak{J}_{16}^{(2)}$.
The remain basis elements taken from $J_{d_{16}^+,T}^{(2)}$.
The numbers written in the 2nd column in Table \ref{tab_g2n16} 
are the number the basis elements coming from $\mathfrak{J}_{8}^{2} $ and $J_{d_{16}^+,T}^{(2)}$, respectively.

\begin{table}[ht]
\centering
\begin{tabular}{ccc}
$|T|$ & $\mathfrak{J}_{16}^{(2)}$ & Total\\
\hline
0 & 1 + 0 & 1\\
1 & 1 + 0 & 1\\
2 & 2 + 1 & 3\\
3 & 2 + 1 & 3\\
4 & 4 + 1 & 5\\
5 & 4 + 1 & 5\\
6 & 5 + 1 & 6\\
7 & 5 + 1 & 6\\
8 & 7 + 1 & 8
\end{tabular}
\caption{Number of basis elements in $\mathfrak{J}_{16}^{(2)}$ attached to $T$}
\label{tab_g2n16}
\end{table}

For $n = 24$, 
we add one more Type II code $p_{24}$ whose generator matrix is
\[
\left(\begin{array}{rrrrrrrrrrrrrrrrrrrrrrrr}
1 & 0 & 0 & 1 & 0 & 1 & 0 & 1 & 0 & 1 & 0 & 1 & 0 & 1 & 0 & 1 & 0 & 1 & 0 & 1 & 0 & 1 & 1 & 0 \\
0 & 1 & 0 & 1 & 0 & 1 & 0 & 1 & 0 & 1 & 0 & 1 & 0 & 0 & 0 & 0 & 0 & 0 & 0 & 0 & 0 & 0 & 1 & 1 \\
0 & 0 & 1 & 1 & 0 & 0 & 0 & 0 & 0 & 0 & 0 & 0 & 0 & 1 & 0 & 1 & 0 & 1 & 0 & 1 & 0 & 1 & 0 & 1 \\
0 & 0 & 0 & 0 & 1 & 1 & 0 & 0 & 0 & 0 & 0 & 0 & 0 & 1 & 0 & 1 & 0 & 1 & 0 & 1 & 0 & 1 & 0 & 1 \\
0 & 0 & 0 & 0 & 0 & 0 & 1 & 1 & 0 & 0 & 0 & 0 & 0 & 1 & 0 & 1 & 0 & 1 & 0 & 1 & 0 & 1 & 0 & 1 \\
0 & 0 & 0 & 0 & 0 & 0 & 0 & 0 & 1 & 1 & 0 & 0 & 0 & 1 & 0 & 1 & 0 & 1 & 0 & 1 & 0 & 1 & 0 & 1 \\
0 & 0 & 0 & 0 & 0 & 0 & 0 & 0 & 0 & 0 & 1 & 1 & 0 & 1 & 0 & 1 & 0 & 1 & 0 & 1 & 0 & 1 & 0 & 1 \\
0 & 0 & 0 & 0 & 0 & 0 & 0 & 0 & 0 & 0 & 0 & 0 & 1 & 1 & 0 & 0 & 0 & 0 & 0 & 0 & 0 & 0 & 1 & 1 \\
0 & 0 & 0 & 0 & 0 & 0 & 0 & 0 & 0 & 0 & 0 & 0 & 0 & 0 & 1 & 1 & 0 & 0 & 0 & 0 & 0 & 0 & 1 & 1 \\
0 & 0 & 0 & 0 & 0 & 0 & 0 & 0 & 0 & 0 & 0 & 0 & 0 & 0 & 0 & 0 & 1 & 1 & 0 & 0 & 0 & 0 & 1 & 1 \\
0 & 0 & 0 & 0 & 0 & 0 & 0 & 0 & 0 & 0 & 0 & 0 & 0 & 0 & 0 & 0 & 0 & 0 & 1 & 1 & 0 & 0 & 1 & 1 \\
0 & 0 & 0 & 0 & 0 & 0 & 0 & 0 & 0 & 0 & 0 & 0 & 0 & 0 & 0 & 0 & 0 & 0 & 0 & 0 & 1 & 1 & 1 & 1
\end{array}\right).
\]
The code $p_{24}$ is the first code written on \cite{munemasa}. 
The basis elements of $\mathfrak{J}^{(2)}_{24}$ are obtained by $\mathfrak{J}^{(2)}_8\, \mathfrak{J}^{(2)}_8 \, \mathfrak{J}^{(2)}_8$, $J^{(2)}_{d_{24}^+,T}$, and $J_{g_{24},T}^{(2)}$.
The numbers written in 2nd column in 
Table \ref{tab_g2n24} are the number of basis elements coming from 
$\mathfrak{J}^{(2)}_8\mathfrak{J}^{(2)}_8\mathfrak{J}^{(2)}_8$, $J^{(2)}_{d_{24}^+,T}$, $J_{g_{24},T}^{(2)}$, and $J_{p_{24},T}^{(2)}$, respectively.
The dimension of $\mathfrak{J}^{(2)}_{24}$ that we can obtain is~$455$
which is equal to the number given in Theorem~\ref{Thm:MolienJac21}.

\begin{table}[ht]
\centering
\begin{tabular}{ccr}
$|T|$ & $\mathfrak{J}_{24}^{(2)}$ & \text{Total}\\
\hline
0  & 1  + 1 + 1 + 0 & 3\\
1  & 1  + 1 + 1 + 0 & 3\\
2  & 3  + 2 + 1 + 0 & 6\\
3  & 5  + 2 + 1 + 1 & 9\\
4  & 8  + 3 + 1 + 1 & 13\\
5  & 10 + 3 + 1 + 1 & 15\\
6  & 15 + 5 + 0 + 1 & 21\\
7  & 16 + 5 + 0 + 1 & 22\\
8  & 20 + 6 + 0 + 1 & 27\\
9  & 22 + 6 + 0 + 1 & 29\\
10 & 24 + 5 + 1 + 1 & 31\\
11 & 24 + 5 + 1 + 1 & 31\\
12 & 27 + 6 + 1 + 1 & 35
\end{tabular}
\caption{Number of basis elements in $\mathfrak{J}_{24}^{(2)}$ attached to $T$}
\label{tab_g2n24}
\end{table}

From the dimension formula of the invariant ring~$\mathfrak{J}^{(2)}$
given in Theorem~\ref{Thm:MolienJac21}, 
we can find that there are~$8$ algebraically independent generators. 
Based on our computation, for $n=8,24$
the generators with this characteristics are: 
 
\[
J_{d_8^+, \emptyset},
J_{d_8^+, [8]},
J_{d_{24}^+, \emptyset},
J_{d_{24}^+, [24]},
J_{g_{24}, \emptyset},
J_{g_{24}, [24]}
\]

\begin{rem}
It is well-known from the works~\cite{Duke, Oura, Runge}
that the weight enumerators of the codes $d_8^{+}$ and $g_{24}$ 
in genus~$2$
are enough to generate the invariant ring
of weight enumerators  of the 
Type~$\II$ codes in genus~$2$ up to~$n = 24$. 
However, our computations shows that
the Jacobi polynomials of $d_8^{+}$ and $g_{24}$
are inadequate to generate the invariant ring~$\mathfrak{J}^{(2)}_{24}$.
Specifically, the Jacobi polynomials of code $p_{24}$ are also required.
\end{rem}

\begin{rem}
The Assmus--Mattson theorem is one of the most important theorems in design and coding theory. By this theorem, we show that all the shells of $d_8^+$ are 3-designs, as proved in Section 6.1 using an invariant-theoretic approach. This suggests that the Assmus--Mattson theorem might be provable using invariant theory.
For related results, 
see
\cite{{BM1},{BM2},{BMY},{extremal design H-M-N},{Miezaki2},{MMN},{extremal design2 M-N}}. 
\end{rem}

\section*{Declaration of competing interest}

The authors declare that 
they have no known competing financial interests or personal relationships 
that could have appeared to influence the work reported in this paper.

\section*{Acknowledgements}
The authors would also like to thank the anonymous
reviewers for their valuable comments on a previous version of the manuscript.
This work was supported by JSPS KAKENHI Grant Numbers 22K03277,
24K06827 
and
SUST Research Centre under 
Project ID
PS/2023/1/22.

\section*{Data availability statement}
The data that support the findings of this study are available from
the corresponding author.

\end{document}